\documentclass[11pt]{amsproc}
\usepackage{graphicx, amssymb, float, url, longtable}
\usepackage[margin=2.5cm]{geometry}

 \newtheorem{theorem}{Theorem}[section]

 \newtheorem{proposition}[theorem]{Proposition}
\newtheorem{example}[theorem]{Example}
\newtheorem{remark}[theorem]{Remark}

\newcommand\PG{\mathsf{PG}}
\newcommand\F{\mathbb{F}}

\renewcommand\le{\leqslant}

\title[ternary two-weight codes]{An enumeration of certain projective ternary two-weight codes and their relationship to the cubic Segre variety} 
\author{Michael Martis}
\address{ %
Centre for the Mathematics of Symmetry and Computation\\
School of Mathematics and Statistics\\
The University of Western Australia\\
35 Stirling Highway, Crawley, W.A. 6009, Australia.}
\email{20930496@student.uwa.edu.au}

\author{John Bamberg}
\address{ %
Centre for the Mathematics of Symmetry and Computation\\
School of Mathematics and Statistics\\
The University of Western Australia\\
35 Stirling Highway, Crawley, W.A. 6009, Australia.}
\email{John.Bamberg@uwa.edu.au}

\author{Sylvia Morris}
\address{ %
Centre for the Mathematics of Symmetry and Computation\\
School of Mathematics and Statistics\\
The University of Western Australia\\
35 Stirling Highway, Crawley, W.A. 6009, Australia.}
\email{aivlysann@gmail.com}

\dedicatory{We have dedicated this work to the memory of Axel Kohnert (1962 -- 2013).}


\begin{document}

\begin{abstract}
We detail the enumeration of all two-intersection sets of the five-dimensional projective space over the field of order $3$ that
are invariant under an element of order $7$, which include the examples of Hill (1973) and Gulliver (1996). Up to projective equivalence, there
are 6635 such two-intersection sets. 
\end{abstract}

\maketitle

\section{Introduction}

In the early 1970's, Philippe Delsarte showed that there is a remarkable connection between three
objects in mathematics:
\begin{enumerate}
\item[(i)] strongly regular graphs;
\item[(ii)] linear codes with two weights;
\item[(iii)] sets of points in a projective space with two intersection sizes with respect to hyperplanes.
\end{enumerate}

To illustrate these connections, consider the five-dimensional projective space over the field of order $3$. There are some
captivating examples of two-intersection sets in this space, one of which was discovered by Ray Hill \cite{Hill73} in 1973 and
it is known as \emph{Hill's $56$-cap} since it consists of $56$ points (of $\PG(5,3)$), no three
lying on a common line. This cap has the property that any hyperplane of $\PG(5,3)$ has only two
possible ways of intersecting the cap: in $11$ or $20$ points. Equivalently, we can associate the points of this cap with the columns of
a generator matrix for a ternary linear code with parameters $[56,6]$ and weights $36$ and $45$.
Using a beautiful geometric construction known as \emph{linear representation}, Hill's $56$-cap gives rise to a geometry known
as a \emph{partial quadrangle}, and it therefore gives rise to a strongly regular graph on $729$
vertices. Today, we also know that Hill's $56$-cap is intimately linked to \empty{Segre's
  hemisystem} of the Hermitian variety $\mathsf{H}(3,3^2)$ and to an interesting imprimitive
cometric $Q$-antipodal $4$-class association scheme \cite{VanDamEtAl}. There are other examples in
this space found by Gulliver \cite{Gulliver1}, and these examples have something in
common; they each admit a common symmetry of order $7$, which happens to be $q^2-q+1$ when
$q=3$. In general, when $q\not\equiv 2 \pmod{3}$, we can \emph{factor out} this symmetry and
identify these two-intersection sets with subsets of the classical algebraic variety known as the
\emph{cubic Segre variety} (see Section \ref{segre}).

This paper details the search for, and enumeration of, a family of strongly regular graphs arising
from two-intersection sets of $\PG(5, 3)$. Of the three mathematical objects above, two-intersection
sets are the most readily computable, and hence effort was focused on finding two-intersection sets
using linear programming techniques. In particular, the projective space $\PG(5, 3)$ was examined as
a successor to $\PG(5,2)$, since the binary projective codes of dimension at most $6$ have been
classified by Bouyukliev \cite{Bouyukliev}. Over 6000 new strongly regular graphs were discovered as
a result of enumerating two-weight ternary codes of dimension $6$, many of which overlap the list of
codes on Kohnert's webpage\footnote{\texttt{http://linearcodes.uni-bayreuth.de/twoweight/}}.
 Up to projective equivalence, there are 6635 such two-intersection sets of $\PG(5,3)$ invariant under a collineation of order $7$.

\section{Correspondences between two-weight linear codes, two-intersection sets and strongly regular graphs}

We begin with the elementary definitions of all three objects, and the relationships between them.
An \emph{$[n, k]$ linear code} is a $k$-dimensional subspace of $\F_q^n$, and its vectors
are known as \emph{codewords}. The weight of a codeword is the number of non-zero values in its
coordinate vector. Every $[n, k]$ linear code is the rowspace of some $k \times n$ generator matrix.
A linear code $C$ is \emph{projective} if no two columns of its generator matrix are linearly
dependent. Equivalently, $C$ is projective if the minimum weight of the dual code $C^\perp$ of $C$
is at least $3$. A two-weight linear code is simply a linear code in which all non-zero codewords
have weight one of two fixed values $w_{1}$ or $w_{2}$. The family of two-weight linear codes is an
important one; a $2$-error-correcting linear code is perfect if and only if its dual is a two-weight
code, and a $1$-error-correcting linear code is uniformly packed if and only if its dual is a
two-weight code \cite[Theorem 4.2, Theorem 4.3]{CalderbankSurvey}.

A two-intersection set
$\mathcal{K}$ with parameters $(n, k, h_{1}, h_{2})$ is a set of points in $\PG(k - 1, q)$ such that
every hyperplane in $\PG(k - 1, q)$ is incident with either $h_{1}$ or $h_{2}$ points in
$\mathcal{K}$. For example, the points lying on a line form a simple two-intersection set with
parameters $(q + 1, k, 1, q + 1)$: every hyperplane either lies completely on the line (and hence
intersects it at $q + 1$ points), or meets it at one point. The complement of a two-intersection set
with parameters $(n, k, h_{1}, h_{2})$ is again a two-intersection set (but with different
parameters).

Every projective two-weight linear code arises from a two-intersection set and gives rise to a
strongly regular graph. A strongly regular graph with parameters $(n, k, \lambda, \mu)$ is a graph
containing $n$ vertices in which every vertex has degree $k$, every two adjacent vertices share
$\lambda$ common neighbours, and every two non-adjacent vertices share $\mu$ common neighbours. The
following results were originally proved by Delsarte \cite{DelsarteSRG}, and provide the necessary
details for converting a two-intersection set into a two-weight linear code and for converting a
two-weight linear code into a two-intersection set or a strongly regular graph.

\begin{theorem}[c.f., Calderbank and Kantor {\cite[Theorem 3.2]{CalderbankSurvey}}]
If $\mathcal{K} = \{ p_{1}, p_{2}, \ldots, p_{n} \}$ is a two-intersection set with parameters $(n,
k, h_{1}, h_{2})$ that spans $\PG(k - 1 , q)$, then
\[
 \mathcal{G} =\begin{bmatrix} p_{1}^{\perp} & p_{2}^{\perp} & \cdots & p_{n}^{\perp}
 \end{bmatrix}\]
  is the generator matrix of a projective two-weight $[n, k]$ linear code with weights $n - h_{1}$
  and $n - h_{2}$, where it is understood that the $p_i$ are vectors representing the homogeneous
  coordinates of points of $\PG(k-1,q)$. Conversely, if $ \mathcal{G} = \begin{bmatrix} c_{1}\,
    c_{2} \cdots c_{n} \end{bmatrix} $ is the generator matrix of a projective two-weight $[n, k]$
  linear code with weights $w_{1}$ and $w_{2}$, then $\mathcal{K} = \{c_{1}^{\perp}, c_{2}^{\perp},
  \ldots, c_{n}^{\perp}\}$ are the homogeneous coordinate vectors for a two-intersection set with
  parameters $(n, k, n - w_{1}, n - w_{2})$ that spans $\PG(k - 1 , q)$.
\end{theorem}

Delsarte defined the associated graph $\Gamma(\mathcal{C})$ of a projective two-weight linear code
$\mathcal{C}$ with weights $w_{1}$ and $w_{2}$ as follows: let the vertices of $\Gamma(\mathcal{C})$
correspond to the codewords of $\mathcal{C}$. Two distinct vertices in $\Gamma(\mathcal{C})$ are
adjacent if and only if the weight of the difference between their corresponding codewords is
$w_{1}$.

\begin{theorem}[Delsarte {\cite[Theorem 2]{DelsarteSRG}}]
 $\Gamma(\mathcal{C})$ is a strongly regular graph for any projective two-weight linear code $\mathcal{C}$. 
 \end{theorem}

%
%

\section{The Segre embedding in $\PG(5, q)$}\label{segre}

Our motivation stemmed from the fact that many of the most interesting two-intersection sets of
$\PG(5, 3)$ had a stabiliser divisible by $7$, and $7$ is a primitive prime divisor of $3^6-1$ (and
so an element of order $7$ must act irreducibly). The Hill $56$-cap and the examples found by
Gulliver could be constructed by taking unions of orbits of this cyclic subgroup. Restricting
the search to two-intersection sets that could be constructed in this manner prevented the scope of
the problem from becoming infeasible and allowed for a full enumeration of possible solutions. To
classify all two-weight ternary codes of dimension $6$ seems to be out of the current range of
computational power.

Let us model $\PG(5,q)$ simply as the field extension $\F_{q^6}^\times$ over
$\F_q^\times$. If $\omega$ is a primitive root of $\F_{q^6}^\times$, then
multiplication by $\omega$ yields a cyclic subgroup acting regularly on the points of $\PG(5,q)$;
the so-called \emph{Singer cycle} of $\PG(5,q)$. If we raise $\omega$ to a suitable power, namely
$(q+1)(q^2+q+1)$, then we will obtain a field element of order $(q-1)(q^2-q+1)$ in
$\F_{q^6}^\times$, which will induce a collineation $\tau$ of $\PG(5,q)$ of order $q^2-q+1$.
Alternatively, we can write the orbits of $\tau$ down by taking the equivalence classes of the
following relation:
\begin{equation} x \sim y \Leftrightarrow x^{(q^2 - q + 1)(q - 1)} = y^{(q^2 - q + 1)(q - 1)}, \label{eqn:equiv_rel} \end{equation}
where $x, y \in \F_{q^6}^\times$.

\begin{remark}
A two-intersection set invariant under $\tau$ will automatically span $\PG(5,q)$ since $\tau$ acts irreducibly
on this space. 
\end{remark}

The map $\sigma : \PG(1, q) \times \PG(2, q) \to \PG(5, q)$ defined by
 \[
 \sigma([X_{1} , X_{2}], [Y_{1} , Y_{2} , Y_{3}]) = [X_{1} Y_{1} , X_{1} Y_{2} , X_{1} Y_{3} , X_{2} Y_{1} , X_{2} Y_{2} , X_{2} Y_{3}]
 \]
 is a \emph{Segre embedding}, and the image of this map
is a cubic \emph{Segre variety} $\mathcal{S}_{1,2}$ (see \cite[Chapter 25]{HirschfeldThas}). In order to work with this Segre embedding, it
became necessary to find an equivalent operation over elements of $\F_{q^6}$.

\begin{proposition} 
The map $(x, y) \mapsto xy$ from $\F_{q^2}^\times \times \F_{q^3}^\times$ to $\F_{q^6}^\times$ is equivalent to the Segre embedding $\sigma$.
\end{proposition}

\begin{proof}
We model $\PG(1, q)$ as $\F_{q^2}$ over $\F_q$ with basis $\{ \chi_{1}, \chi_{2}
\}$, and $\PG(2, q)$ as $\F_{q^3}$ over $\F_q$ with basis $\{ \psi_{1}, \psi_{2},
\psi_{3} \}$.  Now suppose we have $x = [X_{1} , X_{2}] \in \PG(1, q)$ and $y = [Y_{1} , Y_{2} ,
  Y_{3}] \in \PG(2, q)$; that is, $x = X_{1} \chi_{1} + X_{2} \chi_{2}$ and $y = Y_{1} \psi_{1} +
Y_{2} \psi_{2} + Y_{3} \psi_{3}$. \\ Clearly,
\[xy = X_{1} Y_{1} \chi_{1} \psi_{1} +  X_{1} Y_{2} \chi_{1} \psi_{2} + X_{1} Y_{3} \chi_{1} \psi_{3} + X_{2} Y_{1} \chi_{2} \psi_{1} + X_{2} Y_{2} \chi_{2} \psi_{2} + X_{2} Y_{3} \chi_{2} \psi_{3},\] so it suffices to show that $\{\chi_{1} \psi_{1}, \chi_{1} \psi_{2}, \chi_{1} \psi_{3} , \chi_{2} \psi_{1}, \chi_{2} \psi_{2}, \chi_{2} \psi_{3}\}$ is a basis for $\F_{q^6}$ (over $\F_q$). 

To this end, suppose $ \lambda_{1} \chi_{1} \psi_{1} + \lambda_{2} \chi_{1} \psi_{2} + \lambda_{3}
\chi_{1} \psi_{3} + \lambda_{4} \chi_{2} \psi_{1} + \lambda_{5} \chi_{2} \psi_{2} + \lambda_{6}
\chi_{2} \psi_{3} = 0$.  Further suppose, for the sake of contradiction, that $\lambda_{1} \psi_{1}
+ \lambda_{2} \psi_{2} + \lambda_{3} \psi_{3} \neq 0$. We may then rearrange the equation as
follows:
\[ - \chi_{1} \chi_{2}^{-1} = (\lambda_{1} \psi_{1} + \lambda_{2} \psi_{2} + \lambda_{3} \psi_{3})^{-1}(\lambda_{4} \psi_{1} + \lambda_{5} \psi_{2} + \lambda_{6} \psi_{3}).\]

Noting that the left hand side of the equation is an element of $\F_{q^2}$, and the right
hand side of the equation is an element of $\F_{q^3}$, we conclude that both sides of the
equation are elements of $\F_{q^2} \cap \F_{q^3} = \F_q$. Then we must have $\chi_{1} = \lambda \chi_{2}$, for some $\lambda \in \F_q$ -- a contradiction.

Hence, we must have $\lambda_{1} \psi_{1} + \lambda_{2} \psi_{2} + \lambda_{3} \psi_{3} = 0$ and
consequently, $\lambda_{4} \psi_{1} + \lambda_{5} \psi_{2} + \lambda_{6} \psi_{3} = 0$. As
$\psi_{1}$, $\psi_{2}$ and $\psi_{3}$ are linearly independent, we see that $\lambda_{1} =
\lambda_{2} = \cdots = \lambda_{6} = 0$. This confirms that $\{\chi_{1} \psi_{1}, \chi_{1} \psi_{2},
\chi_{1} \psi_{3} , \chi_{2} \psi_{1}, \chi_{2} \psi_{2}, \chi_{2} \psi_{3}\}$ is a basis for
$\F_{q^6}$.
\end{proof} 

It was discovered that, for certain values of $q$, the image of $\sigma$ is a transversal of the orbits discussed above.

\begin{proposition} 
For $q \neq 2 \bmod 3$, every equivalence class of $(\ref{eqn:equiv_rel})$ contains exactly one
element of the Segre variety arising from the products $ \mathcal{S} = \{x y : x \in
\F_{q^2}^\times, y \in \F_{q^3}^\times\}.$
\end{proposition}

\begin{proof} 
Consider the following subgroup of $\F_{q^6}^\times$ of order $(q^2-q+1)(q-1)$:
\[\mathcal{W} := \{ x \in \F_{q^6}^\times : x^{(q^2-q+1)(q-1)} = 1 \}.\] 
We will first show that the subgroup $\mathcal{S} \cap \mathcal{W}$ is contained in $\F_q^\times$.

The order of the intersection $\mathcal{S} \cap \mathcal{W}$ must divide the order of both subgroups, and hence must also divide 
\[\gcd((q-1)(q+1)(q^2+q+1), (q^2-q+1)(q-1)) = (q-1) \gcd((q+1)(q^2+q+1), q^2-q+1). \] 
Noting that $q^2 + q + 1$ and $q^2 - q + 1$ are coprime, we may simplify the expression to $(q-1)
\gcd(q+1, q^2-q+1).$ This can be further simplified to $(q-1) \gcd(q + 1, 3)$, with one step of the
Euclidean algorithm.  Clearly, when $q \neq 2 \bmod 3$, the above expression evaluates to $q -
1$. In these cases, the order of $\mathcal{S} \cap \mathcal{W}$ divides the order of
$\F_q^\times$, and hence $\mathcal{S} \cap \mathcal{W}$ is a subgroup of
$\F_q^\times$.

Now we may proceed to proving the stated result.  Suppose $x, x' \in \F_{q^2}^\times$ and
$y, y' \in \F_{q^3}^\times$. We show that if $xy$ and $x'y'$ are in the same equivalence
class of $(1)$, they represent the same element of the Segre variety. To do so, we suppose that
$(xy)^{(q^2 - q + 1)(q - 1)} = (x'y')^{(q^2 - q + 1)(q - 1)}$.  This equation can be rearranged into
the form $[(x'x^{-1})(y'y^{-1})]^{(q^2 - q + 1)(q - 1)} = 1$.

So $(x'x^{-1})(y'y^{-1})$ is an element of $\mathcal{S} \cap \mathcal{W}$, and is therefore in
$\F_q^\times$, using the fact shown above.  Hence, we have $(x'x^{-1})(y'y^{-1}) = \lambda$,
for some $\lambda \in \F_q^\times$, which may be rearranged into the form $x'y' = \lambda x
y$. Therefore, $xy$ and $x'y'$ are projectively equivalent. We have shown, then, that the Segre variety arising from $\mathcal{S}$ is a transversal of the given
equivalence classes, as there are an equal number of equivalence classes and elements of
$\mathcal{S}$.
\end{proof}

This relationship between the Segre variety and the orbits of a collineation group allows for the
expression of conforming two-intersection sets as geometric entities in the Segre variety. Every
orbit included in such a two-intersection set corresponds to exactly one point in the Segre
variety. 

\begin{example}
Fix a primitive root $\omega$ of $\F_{3^6}$. Then we can describe the eight orbits
comprising the Hill $56$-cap by single elements of the Segre variety:
\[
1, \omega^{91}, \omega^{14}, \omega^{42}, \omega^{105}, \omega^{126}, \omega^{133}, \omega^{217}.
\]
These elements of the Segre variety were obtained by taking the products $\omega^a\omega^b$, where 
\[
(a,b)\in \{ (0, 0), (0, 91), (560, 182), (588, 182), (560, 273), (672, 182), (588, 273), (672,
273)\}.\] Note that the first coordinates of these pairs are each divisible by $28$, and the second
coordinates are each divisible by $91$ (and so $\omega^a\in \F_{q^3}$ and $\omega^b\in
\F_{q^2}$ for each $(a,b)$).  We will catalogue each of our examples in Section
\ref{results} by points of the Segre variety in this way.
\end{example}

\section{Necessary conditions for the existence of two-intersection sets}

The following is an extension of the results presented by Penttila and Royle
\cite[\S2]{RoyleParameters}, and was used to quickly identify implausible parameter sets in the
computational search.

\begin{proposition} 
Suppose $\mathcal{K}$ is a two-intersection set with parameters $(n, k, h_{1}, h_{2})$. Then there
must exist an integer solution to:
\begin{equation} n^2 \frac {q ^ {k - 2} - 1} {q - 1} + n ( (1 - h_{1} - h_{2}) \frac {q^{k - 1} - 1} {q - 1} - \frac {q^{k - 2} - 1} {q - 1} ) + h_{1} h_{2} \frac {q^k - 1} {q - 1} = 0. \end{equation}
 \end{proposition} 

\begin{proof} 
If $t_{1}$ and $t_{2}$ are the number of hyperplanes that are incident with $h_{1}$ and $h_{2}$
points in $\mathcal{K}$, respectively, then counting arguments reveal the following equations:
\begin{center}
\begin{tabular}{l|p{7cm}}
\hline\\ 
$ t_{1} + t_{2} = (q^{k} - 1)/(q - 1)$ & counting the number of hyperplanes\\[15pt] \hline\\
$h_{1} t_{1} + h_{2} t_{2} = n (q^{k-1} - 1)(q -1)$& counting incident point-hyperplane pairs \\[15pt] \hline\\
$h_{1} (h_{1} - 1) t_{1} + h_{2} (h_{2} - 1) t_{2} = n (n - 1) (q ^ {k - 2} - 1)/(q - 1)$ &\begin{minipage}{7cm}counting triples
  $(p_1,p_2,\Pi)$ where $p_1$ and $p_2$ are two different points incident with the hyperplane $\Pi$ \end{minipage}\\[20pt] \hline
\end{tabular}
\end{center}
If we take the left-hand sides of these three equations, and sum them with coefficients $h_1h_2$,
$1-h_1-h_2$, and $1$ accordingly, we will obtain $0$. The result then follows.
\end{proof} 

\begin{proposition}
Suppose $\mathcal{K}$ is a two-intersection set with parameters $(n, k, h_{1}, h_{2})$. Then $h_{2}
- h_{1}$ must divide $q^{k - 2}$.
\end{proposition}

\begin{proof}
Suppose that $\rho_{1}$ and $\rho_{2}$ are the number of hyperplanes through a point $P \in
\mathcal{K}$ that are incident with $h_{1}$ and $h_{2}$ points in $\mathcal{K}$,
respectively. Similarly, suppose $\sigma_{1}$ and $\sigma_{2}$ are the number of hyperplanes through
a point $Q \notin \mathcal{K}$ that are incident with $h_{1}$ and $h_{2}$ points in $\mathcal{K}$,
respectively. Then counting arguments show that:
\begin{equation} \rho_{1} + \rho_{2} = \frac{q^{k - 1} - 1}{q - 1}, \end{equation}
\begin{equation} \sigma_{1} + \sigma_{2} = \frac{q^{k - 1} - 1}{q - 1}, \end{equation}
\begin{equation} (h_{1} - 1)\rho_{1} + (h_{2} - 1)\rho_{2} = (n - 1)\frac{q^{k - 2} - 1}{q - 1}, \text{ and} \end{equation}
\begin{equation} h_{1}\sigma_{1} + h_{2}\sigma_{2} = n\frac{q^{k - 2} - 1}{q - 1}. \end{equation}
By solving this system of equations, we see that $\rho_{2} - \sigma_{2} = \frac{q^{k - 2}}{h_{2} - h_{1}}$ must be an integer. 
\end{proof} 

\section{A computational search for two-intersection sets}

A significant advantage of searching for two-intersection sets, rather than the other mathematical
objects discussed in this paper, is the relative ease with which two-intersection sets can be found
computationally. The task of finding two-intersection sets in a projective space can be naturally
represented as a linear programming problem, as evinced by Kohnert \cite{KohnertCodes}. This section
details a method of framing the task in such a way.

Let $G$ be a group of collineations of $\PG(5,q)$.
We let $M$ be the $G$-quotient of the point-hyperplane incidence matrix wherein the $(i, j)^{th}$ entry is the number of
elements from the $i$th $G$-orbit on points incident with the $j$th $G$-orbit representative on hyperplanes.
Alternatively, we can use a duality arising from a bilinear form, say, so that the hyperplanes are replaced by points
and incidence is replaced by orthogonality. In our application, where $G$ is a cyclic group of order $q^2-q+1$ acting
on $\PG(5,q)$, we can simply construct $M$ by indexing its rows and columns by the cubic Segre variety $\mathcal{S}_{1,2}$.
The relative trace map (from $\F_{q^6}$ to $\F_q$) yields a bilinear form when we model $\PG(5,q)$ by the field $\F_{q^6}$:
\[
B(x,y):=\mathrm{Tr}_{q^6\to q}(xy).
\]
Hence, we may instead calculate $M(i,j)$ by calculating the number of elements $g\in G$ such that $B(s_i g,s_j)=0$,
where $s_i,s_j$ are the $i$-th and $j$-th elements of the Segre variety modelled in $\F_{q^6}$.

An \emph{orbit inclusion vector} is a 0-1 vector in which the $i$th entry is $1$ or $0$, dependent on
whether the $i$-th orbit is included or excluded from the current point set (resp.). Clearly, premultiplying
the inclusion vector with the incidence matrix yields a vector wherein the $i$-th entry
contains the number of points in the current point set incident with the $i$-th hyperplane.

In order to find a two-intersection set with intersection values $h_{1}$ and $h_{2}$, we need
to find solutions to the system of equations:
\[ \textbf{x} M = ( h_{1} \text{ or } h_{2},  \quad h_{1} \text{ or } h_{2},\quad \ldots ,\quad h_{1} \text{ or } h_{2} ), \]
where $\textbf{x}$ is a 0-1 vector. While this is not a linear system, it can be converted into one
with a minor modification. For a given solution $\textbf{x}$, we introduce a vector $\textbf{y}$ of
length $n$, such that the $i$th entry of $\textbf{y}$ is $0$ if the $i$th hyperplane is incident
with $h_{1}$ points, or $1$ if the $i$th hyperplane is incident with $h_{2}$ points. This allows the
problem to be framed as a linear system of equations:
\[ \textbf{x} M + (h_{1} - h_{2})  \textbf{y} = \left( h_{1}, h_{1}, \ldots , h_{1}  \right), \]
which can then be solved with software such as \textit{Gurobi} \cite{gurobi}. This also has the benefit of immediately revealing 
which hyperplanes are incident with $h_{1}$ points and which are incident with $h_{2}$.

\section{Optimising the search}

An initial search for
two-intersection sets, filtered only by fulfilment of the necessary conditions discussed above, was
able to return hundreds of thousands of results without exhausting the problem space.
To prevent the re-computation of redundant data, the action of the full collineation group was used to impose a number of
constraints on the search. We computed valuable information regarding the symmetries of the projective space
using the free software package \textsc{GAP} \cite{GAP4}. Firstly, it was noted that the normaliser of our element $\tau$ of order $7$
in the full collineation group acted transitively on the point-orbits of $\tau$, meaning any solution
was equivalent to one that included the first orbit. Hence, the first orbit was selected in every
one of the reported solutions.

Next, the subgroup that stabilised the first orbit was examined. This group had nine orbits itself,
represented by the orbits with indices in the set $\{1, 2, 3, 5, 6, 17, 18, 20, 24\}$. Hence, any
solution that included two or more orbits was equivalent to one that included either orbits 1 and 2,
or orbits 1 and 3, or orbits 1 and 5 etc. The solution space was then restricted to ensure that one
of these orbit pairs was always included.

\begin{figure}
  \begin{center}
    \includegraphics[width=0.98\textwidth]{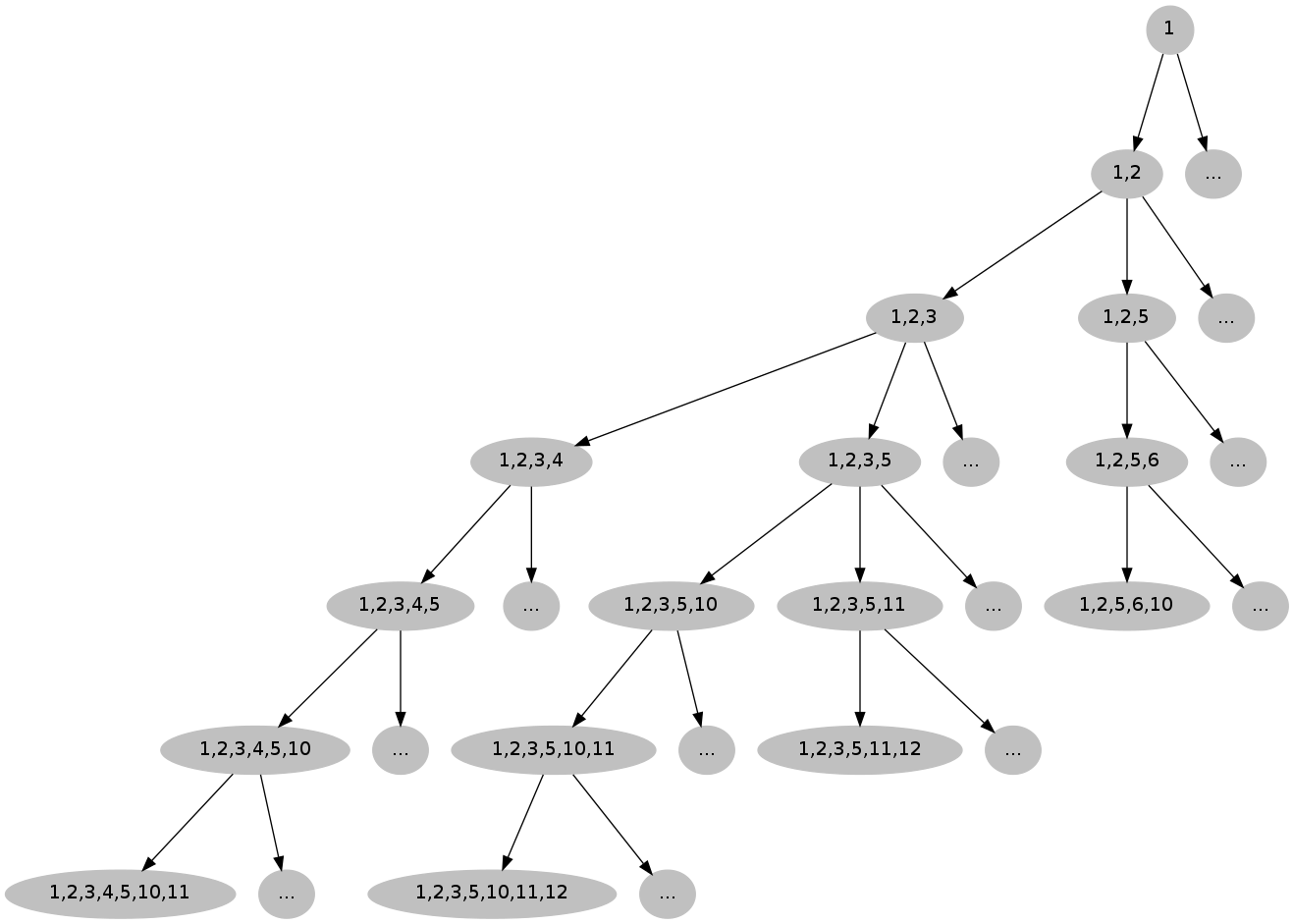}
    \caption{The tree of orbits used to restrict the computational search space. Each node labeled
      with ellipses denotes at least six siblings not shown for the sake of brevity. }
  \end{center}
\end{figure}

This method was reapplied to those subproblems that proved computationally difficult. For instance,
including orbits 1 and 2 in a solution was shown to be equivalent to including one of nine triples
of orbits. In this manner, a tree of computations was formed and traversed, sometimes reaching seven
levels deep.

One optimisation that proved particularly efficient was facilitated by the use of the branching
method discussed above. As noted, there are many sets of orbits that are equivalent to one
another. The search was accelerated by precluding the consideration of any set equivalent to one
already encountered. For example, there exist 52 pairs of orbits whose inclusion is equivalent to
the inclusion of the orbits 1 and 2. Hence, after all solutions containing orbits 1 and 2 have been
found, any of these 52 equivalent pairs of orbits can be excluded from appearing again in any future
results.

These optimisation techniques allowed for the enumeration of all $\tau$-invariant two-intersection sets in $\PG(5,
3)$ within a viable time-frame (approximately 2 weeks).

In order to have confidence in the correctness of the search, we used the constraint satisfaction solver \textsc{Minion} \cite{minion} to
simulate some of our results. In particular, we enumerated all the two-intersection sets with parameters
$\{11,20\}$, $\{21,30\}$, $\{26,35\}$, $\{28,37\}$, or $\{31,40\}$, invariant under a group with order divisible by $7$, and we obtained the
same results which we outline in Table \ref{smallparams}:
\begin{table}[H]
\begin{tabular}{c|p{5cm}|p{4cm}}
\hline
Parameters  & \begin{minipage}{5cm}Number of two-intersection\\ sets up to equivalence\end{minipage} & \begin{minipage}{4cm}Stabiliser sizes\\ (up to multiplicity)\end{minipage}\\
\hline
$\{11,20\}$ & 1 &  $40320$\\
$\{21,30\}$ & 6 & $7^4, 42^2$\\
$\{26,35\}$ & 4 & $7^2, 21^2$\\
$\{28,37\}$ & 22 & $7^9, 14^3, 21^2, 42^6, 168^1, 546^1$\\
$\{31,40\}$ & 55 & $7^{46}, 14^8, 26127360^1$\\
\hline
\end{tabular}
\caption{Examples for small parameter sets.}\label{smallparams}
\end{table}

More evidence for the efficacy of our computation is the third
author's\footnote{Sylvia Morris, `Symplectic translation planes, pseudo-ovals, and maximal $4$-arcs', dissertation for the Master of Philosophy (Research) of The
University of Western Australia (2014).} enumeration of the spreads of the symplectic polar space $\mathsf{W}(5,q)$, for small $q$.
In particular, the results were previously known for $q\le 4$, which verified the computational data obtained.

\section{Results}\label{results}

The following results were found using the methods discussed in the previous two sections. The
``Segre variety representation'' column lists the points in the Segre variety that correspond to the
given two-intersection set, in the following format: each entry specifies a pair of exponents to the
primitive root $\omega$ of $\F_{3^6}$, which is represented as $\F_3[x]/{\langle x^6 - x^4 +
  x^2 - x - 1\rangle}$. The specific element of the Segre variety can be obtained by applying the
multiplication map to the resulting elements. That is, the map $(a, b) \mapsto \omega^a \omega^b$ yields
the correct element of the Segre variety in the model $\F_{3^6}^\times$ over $\F_{3}^\times$. The
``stabiliser size'' column lists the size of the largest group that stabilises the given orbit
set. It is important to note that not all
two-intersection sets, and hence not all strongly regular graphs, are present in the tables that
follow: 6635 solutions were found altogether, with all but the 187 listed below having a full stabiliser of order $7$ in the
collineation group. Every one of the 6635 computed two-intersection sets was found to
correspond to a unique strongly regular graph. This was achieved by using the software \texttt{nauty} 
\cite{nauty} to identify the canonical representation of the graph. The full list of two-intersection sets
can be found at the webpage \url{https://researchdataonline.research.uwa.edu.au/handle/123456789/1358}.

{\footnotesize
\begin{longtable}{|p{0.85\textwidth}|p{0.1\textwidth}|}
\caption{Two-intersection sets of type $(56, 6, 11, 20)$} \\ 
\hline
\textbf{Segre variety representation} & \textbf{Stabiliser} \\
& \textbf{size}  \\
\hline
\endhead
$(0, 0)$, $(0, 91)$, $(560, 182)$, $(588, 182)$, $(560, 273)$, $(672, 182)$, $(588, 273)$, $(672, 273)$ & 40320   \\ \hline 
\end{longtable} }
The only two-intersection set of size 56 is the Hill cap \cite{Hill73}, which gives rise to a strongly regular graph with parameters $(729, 112, 1, 20)$ (this is also Example FE2 in \cite[\S 11]{CalderbankSurvey}).

{\footnotesize
\begin{longtable}{|p{0.85\textwidth}|p{0.1\textwidth}|}
\caption{Two-intersection sets of type $(84, 6, 21, 30)$} \\ 
\hline
\textbf{Segre variety representation} & \textbf{Stabiliser} \\
& \textbf{size}  \\\hline
\endhead
$(0, 0)$, $(0, 182)$, $(168, 0)$, $(588, 364)$, $(672, 91)$, $(588, 182)$, $(28, 91)$, $(56, 91)$, $(448, 455)$, $(476, 455)$, $(560, 455)$, $(168, 182)$ & 42 \\ \hline 
$(0, 0)$, $(0, 182)$, $(504, 364)$, $(560, 364)$, $(560, 182)$, $(588, 273)$, $(672, 273)$, $(700, 273)$, $(28, 273)$, $(504, 546)$, $(56, 273)$, $(448, 637)$ & 42 \\ \hline 
\end{longtable}}
The above two-intersection sets give rise to strongly regular graphs with parameters $(729, 168, 27, 42)$.
Gulliver's $(84,6,54)$-code \cite{Gulliver1} has a stabiliser of order $7$ so does not appear in this list.

{\footnotesize
\begin{longtable}{|p{0.85\textwidth}|p{0.1\textwidth}|}
\caption{Two-intersection sets of type $(98, 6, 26, 35)$} \\ 
\hline
\textbf{Segre variety representation} & \textbf{Stabiliser} \\
& \textbf{size}  \\\hline
\endhead
$(0, 0)$, $(0, 91)$, $(28, 0)$, $(56, 0)$, $(448, 364)$, $(504, 364)$, $(168, 0)$, $(252, 0)$, $(252, 546)$, $(560, 273)$, $(588, 273)$, $(28, 182)$, $(672, 273)$, $(448, 546)$ & 21 \\ \hline 
$(0, 0)$, $(28, 0)$, $(504, 364)$, $(672, 364)$, $(168, 637)$, $(560, 273)$, $(28, 91)$, $(588, 273)$, $(252, 637)$, $(448, 455)$, $(700, 273)$, $(168, 91)$, $(560, 455)$, $(448, 637)$ & 21 \\ \hline 
\end{longtable}
}
The above two-intersection sets give rise to strongly regular graphs with parameters $(729, 196, 43, 56)$.
Gulliver's $(98,6,63)$-code \cite{Gulliver1} has a stabiliser of order $7$ so does not appear in this list.

{\footnotesize
\begin{longtable}{|p{0.85\textwidth}|p{0.1\textwidth}|}
\caption{Two-intersection sets of type $(91, 6, 28, 37)$} \\ 
\hline
\textbf{Segre variety representation} & \textbf{Stabiliser} \\
& \textbf{size}  \\\hline
\endhead
$(0, 0)$, $(0, 91)$, $(0, 182)$, $(0, 273)$, $(56, 0)$, $(252, 0)$, $(672, 364)$, $(700, 364)$, $(588, 182)$, $(252, 546)$, $(672, 182)$, $(448, 455)$, $(448, 637)$ & $14$ \\ \hline 
$(0, 0)$, $(0, 91)$, $(0, 273)$, $(504, 364)$, $(560, 364)$, $(504, 273)$, $(560, 273)$, $(476, 455)$, $(672, 273)$, $(504, 455)$, $(700, 273)$, $(560, 455)$, $(56, 273)$ & $168$ \\ \hline 
$(0, 0)$, $(0, 91)$, $(0, 182)$, $(0, 273)$, $(476, 364)$, $(280, 0)$, $(700, 364)$, $(280, 546)$, $(28, 91)$, $(700, 182)$, $(448, 455)$, $(476, 546)$, $(252, 91)$ & $42^*$ \\ \hline 
$(0, 0)$, $(28, 0)$, $(280, 0)$, $(560, 182)$, $(476, 273)$, $(588, 273)$, $(56, 91)$, $(700, 182)$, $(252, 637)$, $(448, 455)$, $(672, 273)$, $(168, 91)$, $(504, 546)$ & $21$ \\ \hline 
$(0, 0)$, $(28, 0)$, $(56, 0)$, $(448, 364)$, $(504, 364)$, $(168, 0)$, $(560, 364)$, $(588, 364)$, $(280, 0)$, $(700, 364)$, $(252, 546)$, $(672, 182)$, $(476, 546)$ & $42$ \\ \hline 
$(0, 0)$, $(28, 0)$, $(56, 0)$, $(448, 364)$, $(168, 0)$, $(280, 0)$, $(560, 182)$, $(588, 182)$, $(252, 546)$, $(672, 182)$, $(700, 182)$, $(476, 546)$, $(504, 546)$ & $42$ \\ \hline 
$(0, 0)$, $(28, 0)$, $(56, 0)$, $(476, 364)$, $(504, 364)$, $(168, 0)$, $(560, 364)$, $(588, 364)$, $(280, 0)$, $(252, 546)$, $(672, 182)$, $(700, 182)$, $(448, 546)$ & $42$ \\ \hline 
$(0, 0)$, $(28, 0)$, $(448, 364)$, $(168, 0)$, $(252, 0)$, $(672, 364)$, $(560, 182)$, $(588, 182)$, $(280, 546)$, $(700, 182)$, $(56, 182)$, $(476, 546)$, $(504, 546)$ & $14$ \\
\hline 
\end{longtable}}
The above two-intersection sets give rise to strongly regular graphs with parameters $(729, 182, 55, 42)$.
The three examples in the table above that have an asterisk in the the last column arise from partial spreads of $\PG(5,3)$; the `SU2' construction in \cite[\S7]{CalderbankSurvey}.
One of these examples also arises from the `CY1' construction in \cite[\S9]{CalderbankSurvey}, denoted by two asterisks.

{\footnotesize
\begin{longtable}{|p{0.85\textwidth}|p{0.1\textwidth}|}
\caption{Two-intersection sets of type $(112, 6, 31, 40)$} \\ 
\hline
\textbf{Segre variety representation} & \textbf{Stabiliser} \\
& \textbf{size}  \\\hline
\endhead
$(0, 0)$, $(0, 91)$, $(0, 182)$, $(0, 273)$, $(28, 0)$, $(56, 0)$, $(476, 364)$, $(280, 455)$, $(476, 273)$, $(252, 546)$, $(280, 546)$, $(168, 91)$, $(28, 273)$, $(56, 273)$, $(252, 91)$, $(168, 182)$ & 14 \\ \hline 
$(0, 0)$, $(0, 91)$, $(0, 182)$, $(0, 273)$, $(56, 0)$, $(280, 0)$, $(252, 546)$, $(28, 91)$, $(280, 637)$, $(476, 455)$, $(28, 182)$, $(168, 91)$, $(476, 546)$, $(56, 273)$, $(252, 91)$, $(168, 182)$ & 14 \\ \hline 
$(0, 0)$, $(0, 91)$, $(0, 182)$, $(0, 273)$, $(560, 364)$, $(588, 364)$, $(252, 0)$, $(280, 0)$, $(672, 364)$, $(560, 273)$, $(588, 273)$, $(56, 91)$, $(252, 637)$, $(280, 637)$, $(672, 273)$, $(56, 182)$ & 14 \\ \hline 
$(0, 0)$, $(0, 91)$, $(0, 182)$, $(0, 273)$, $(476, 364)$, $(560, 364)$, $(252, 0)$, $(560, 182)$, $(504, 273)$, $(700, 91)$, $(252, 546)$, $(252, 637)$, $(672, 273)$, $(504, 455)$, $(168, 91)$, $(476, 546)$ & 14 \\ \hline 
$(0, 0)$, $(0, 91)$, $(0, 182)$, $(0, 273)$, $(28, 0)$, $(476, 364)$, $(168, 0)$, $(476, 273)$, $(168, 637)$, $(28, 91)$, $(476, 455)$, $(28, 182)$, $(168, 91)$, $(476, 546)$, $(28, 273)$, $(168, 182)$ & 26127360  \\ \hline 
$(0, 0)$, $(0, 91)$, $(0, 182)$, $(56, 0)$, $(168, 0)$, $(588, 364)$, $(280, 455)$, $(476, 273)$, $(588, 182)$, $(28, 91)$, $(252, 637)$, $(448, 455)$, $(672, 273)$, $(56, 182)$, $(560, 455)$, $(168, 182)$ & 14 \\ \hline 
$(0, 0)$, $(0, 91)$, $(0, 182)$, $(504, 364)$, $(168, 0)$, $(672, 364)$, $(700, 91)$, $(28, 91)$, $(672, 182)$, $(448, 455)$, $(476, 455)$, $(560, 455)$, $(588, 455)$, $(504, 546)$, $(252, 91)$, $(168, 182)$ & 14 \\ \hline 
$(0, 0)$, $(0, 182)$, $(504, 364)$, $(560, 364)$, $(672, 364)$, $(560, 182)$, $(700, 91)$, $(168, 637)$, $(672, 182)$, $(588, 273)$, $(56, 91)$, $(448, 455)$, $(504, 455)$, $(28, 273)$, $(504, 546)$, $(252, 91)$ & 14 \\
\hline 
\end{longtable}}
The above two-intersection sets give rise to strongly regular graphs with parameters $(729, 224, 61, 72)$.
The example above with stabiliser of order 26127360 is none-other than the points of the elliptic quadric $\mathsf{Q}^-(5,3)$ (this is also Example RT2 in \cite[\S 10]{CalderbankSurvey}).

{\footnotesize
\begin{longtable}{|p{0.85\textwidth}|p{0.1\textwidth}|}
\caption{Two-intersection sets of type $(126, 6, 36, 45)$} \\ 
\hline
\textbf{Segre variety representation} & \textbf{Stabiliser} \\
& \textbf{size}  \\\hline
\endhead
$(0, 0)$, $(0, 182)$, $(28, 0)$, $(504, 364)$, $(560, 364)$, $(280, 0)$, $(700, 364)$, $(560, 182)$, $(476, 273)$, $(252, 546)$, $(280, 546)$, $(672, 182)$, $(700, 182)$, $(252, 637)$, $(28, 182)$, $(672, 273)$, $(476, 546)$, $(504, 546)$ & 21 \\ \hline 
$(0, 0)$, $(0, 273)$, $(28, 0)$, $(56, 0)$, $(476, 364)$, $(504, 364)$, $(588, 364)$, $(252, 546)$, $(28, 91)$, $(672, 182)$, $(588, 273)$, $(252, 637)$, $(476, 455)$, $(672, 273)$, $(504, 455)$, $(448, 546)$, $(56, 273)$, $(448, 637)$ & 21 \\ \hline 
$(0, 0)$, $(0, 182)$, $(28, 0)$, $(448, 364)$, $(476, 364)$, $(560, 364)$, $(700, 364)$, $(280, 455)$, $(560, 182)$, $(588, 273)$, $(56, 91)$, $(700, 182)$, $(280, 637)$, $(28, 182)$, $(448, 546)$, $(476, 546)$, $(588, 455)$, $(56, 273)$ & 13063680   \\ \hline 
$(0, 0)$, $(0, 273)$, $(28, 0)$, $(672, 364)$, $(700, 364)$, $(280, 455)$, $(588, 182)$, $(280, 546)$, $(56, 91)$, $(672, 273)$, $(504, 455)$, $(56, 182)$, $(700, 273)$, $(168, 91)$, $(28, 273)$, $(588, 455)$, $(504, 546)$, $(168, 182)$ & 14 \\ \hline 
$(0, 0)$, $(0, 91)$, $(28, 0)$, $(448, 364)$, $(476, 364)$, $(560, 364)$, $(700, 364)$, $(588, 182)$, $(700, 91)$, $(280, 546)$, $(28, 91)$, $(588, 273)$, $(448, 455)$, $(280, 637)$, $(476, 455)$, $(56, 182)$, $(560, 455)$, $(56, 273)$ & 42 \\ \hline 
$(0, 0)$, $(0, 91)$, $(448, 364)$, $(476, 364)$, $(560, 182)$, $(588, 182)$, $(280, 546)$, $(560, 273)$, $(588, 273)$, $(700, 182)$, $(448, 455)$, $(280, 637)$, $(476, 455)$, $(28, 182)$, $(56, 182)$, $(700, 273)$, $(28, 273)$, $(56, 273)$ & 42 \\ \hline 
$(0, 0)$, $(0, 91)$, $(0, 182)$, $(56, 0)$, $(448, 364)$, $(476, 364)$, $(588, 364)$, $(280, 0)$, $(588, 182)$, $(280, 546)$, $(588, 273)$, $(448, 455)$, $(280, 637)$, $(476, 455)$, $(56, 182)$, $(448, 546)$, $(476, 546)$, $(56, 273)$ & 42 \\
\hline 
\end{longtable}}
The above two-intersection sets give rise to strongly regular graphs with parameters $(729, 252, 81, 90)$.
The example above with stabiliser of order 13063680 can be constructed by taking a quadratic form $Q$ of minus type defining an elliptic quadric $\mathsf{Q}^-(5,3)$,
and then considering the $126$ points $p$ such that $Q(p)=1$.

{\footnotesize
\begin{longtable}{|p{0.85\textwidth}|p{0.1\textwidth}|}
\caption{Two-intersecetion sets of type $(140, 6, 41, 50)$} \\ 
\hline
\textbf{Segre variety representation} & \textbf{Stabiliser} \\
& \textbf{size}  \\\hline
\endhead
$(0, 0)$, $(0, 91)$, $(0, 182)$, $(28, 0)$, $(476, 364)$, $(280, 455)$, $(588, 182)$, $(504, 273)$, $(168, 637)$, $(28, 91)$, $(672, 182)$, $(700, 182)$, $(448, 455)$, $(280, 637)$, $(476, 455)$, $(28, 182)$, $(504, 455)$, $(476, 546)$, $(168, 182)$, $(448, 637)$ & 21 \\ \hline 
$(0, 0)$, $(0, 91)$, $(0, 182)$, $(0, 273)$, $(28, 0)$, $(56, 0)$, $(588, 364)$, $(672, 364)$, $(700, 91)$, $(28, 91)$, $(588, 273)$, $(56, 91)$, $(700, 182)$, $(448, 455)$, $(28, 182)$, $(672, 273)$, $(56, 182)$, $(448, 546)$, $(28, 273)$, $(56, 273)$ & 14 \\ \hline 
$(0, 0)$, $(0, 91)$, $(0, 182)$, $(28, 0)$, $(476, 364)$, $(168, 637)$, $(280, 546)$, $(560, 273)$, $(28, 91)$, $(56, 91)$, $(252, 637)$, $(476, 455)$, $(28, 182)$, $(448, 546)$, $(560, 455)$, $(476, 546)$, $(504, 546)$, $(56, 273)$, $(252, 91)$, $(168, 182)$ & 21 \\ \hline 
$(0, 0)$, $(0, 91)$, $(0, 182)$, $(28, 0)$, $(280, 0)$, $(280, 455)$, $(560, 182)$, $(672, 91)$, $(588, 182)$, $(252, 546)$, $(280, 546)$, $(28, 91)$, $(672, 182)$, $(700, 182)$, $(476, 455)$, $(28, 182)$, $(476, 546)$, $(588, 455)$, $(504, 546)$, $(252, 91)$ & 21 \\ \hline 
$(0, 0)$, $(0, 91)$, $(0, 182)$, $(28, 0)$, $(504, 364)$, $(560, 364)$, $(588, 364)$, $(280, 0)$, $(700, 364)$, $(280, 455)$, $(476, 273)$, $(252, 546)$, $(280, 546)$, $(28, 91)$, $(672, 182)$, $(252, 637)$, $(28, 182)$, $(672, 273)$, $(476, 546)$, $(588, 455)$ & 21 \\ \hline 
$(0, 0)$, $(0, 91)$, $(0, 182)$, $(28, 0)$, $(476, 364)$, $(504, 364)$, $(560, 364)$, $(588, 364)$, $(252, 0)$, $(280, 0)$, $(672, 364)$, $(700, 364)$, $(280, 455)$, $(476, 273)$, $(280, 546)$, $(28, 91)$, $(252, 637)$, $(28, 182)$, $(672, 273)$, $(588, 455)$ & 21 \\ \hline 
$(0, 0)$, $(0, 91)$, $(0, 182)$, $(0, 273)$, $(28, 0)$, $(56, 0)$, $(476, 364)$, $(504, 364)$, $(588, 364)$, $(700, 364)$, $(588, 182)$, $(700, 91)$, $(28, 91)$, $(588, 273)$, $(56, 91)$, $(476, 455)$, $(504, 455)$, $(56, 182)$, $(588, 455)$, $(56, 273)$ & 42 \\
\hline 
\end{longtable}}
The above two-intersection sets give rise to strongly regular graphs with parameters $(729, 280, 103, 110)$.

{\footnotesize
\begin{longtable}{|p{0.85\textwidth}|p{0.1\textwidth}|}
\caption{Two-intersection sets of type $(154, 6, 46, 55)$} \\ 
\hline
\textbf{Segre variety representation} & \textbf{Stabiliser} \\
& \textbf{size}  \\\hline
\endhead
$(0, 0)$, $(0, 91)$, $(0, 182)$, $(28, 0)$, $(56, 0)$, $(168, 0)$, $(560, 364)$, $(588, 364)$, $(560, 182)$, $(476, 273)$, $(672, 91)$, $(588, 182)$, $(700, 91)$, $(560, 273)$, $(252, 637)$, $(476, 455)$, $(28, 182)$, $(672, 273)$, $(56, 182)$, $(168, 91)$, $(252, 91)$, $(168, 182)$ & 14 \\ \hline 
$(0, 0)$, $(0, 91)$, $(0, 182)$, $(0, 273)$, $(28, 0)$, $(56, 0)$, $(448, 364)$, $(504, 364)$, $(168, 0)$, $(560, 364)$, $(588, 364)$, $(252, 0)$, $(672, 364)$, $(672, 91)$, $(252, 546)$, $(560, 273)$, $(588, 273)$, $(28, 182)$, $(672, 273)$, $(448, 546)$, $(560, 455)$, $(588, 455)$ & 42 \\ \hline 
$(0, 0)$, $(0, 91)$, $(0, 182)$, $(0, 273)$, $(28, 0)$, $(56, 0)$, $(588, 364)$, $(252, 0)$, $(560, 182)$, $(588, 182)$, $(168, 637)$, $(560, 273)$, $(28, 91)$, $(588, 273)$, $(56, 91)$, $(28, 182)$, $(56, 182)$, $(28, 273)$, $(588, 455)$, $(56, 273)$, $(252, 91)$, $(168, 182)$ & 14 \\ \hline 
$(0, 0)$, $(0, 91)$, $(0, 182)$, $(0, 273)$, $(28, 0)$, $(56, 0)$, $(168, 0)$, $(252, 0)$, $(672, 91)$, $(700, 91)$, $(252, 546)$, $(560, 273)$, $(28, 91)$, $(56, 91)$, $(448, 455)$, $(28, 182)$, $(672, 273)$, $(504, 455)$, $(56, 182)$, $(28, 273)$, $(56, 273)$, $(168, 182)$ & 14 \\ \hline 
$(0, 0)$, $(0, 91)$, $(0, 182)$, $(0, 273)$, $(28, 0)$, $(56, 0)$, $(448, 364)$, $(504, 364)$, $(560, 364)$, $(280, 0)$, $(700, 364)$, $(672, 91)$, $(700, 91)$, $(280, 546)$, $(28, 91)$, $(28, 182)$, $(672, 273)$, $(700, 273)$, $(448, 546)$, $(476, 546)$, $(28, 273)$, $(168, 182)$ & 14 \\ \hline 
$(0, 0)$, $(0, 91)$, $(0, 182)$, $(0, 273)$, $(28, 0)$, $(56, 0)$, $(476, 364)$, $(252, 0)$, $(280, 0)$, $(280, 455)$, $(672, 91)$, $(252, 546)$, $(168, 637)$, $(280, 546)$, $(560, 273)$, $(56, 91)$, $(280, 637)$, $(476, 455)$, $(28, 182)$, $(56, 182)$, $(560, 455)$, $(476, 546)$ & 14 \\
\hline 
\end{longtable}}
The above two-intersection sets give rise to strongly regular graphs with parameters $(729, 308, 127, 132)$.
These examples are different to Gulliver's $(154,6,99)$-code \cite{Gulliver2}, since his example has a stabiliser of order 11.

{\footnotesize
\begin{longtable}{|p{0.85\textwidth}|p{0.1\textwidth}|}
\caption{Two-intersection sets of type $(168, 6, 51, 60)$} \\ 
\hline
\textbf{Segre variety representation} & \textbf{Stabiliser} \\
& \textbf{size}  \\\hline
\endhead
$(0, 0)$, $(0, 91)$, $(0, 182)$, $(0, 273)$, $(28, 0)$, $(168, 0)$, $(560, 364)$, $(588, 364)$, $(672, 364)$, $(476, 273)$, $(672, 91)$, $(280, 546)$, $(28, 91)$, $(672, 182)$, $(280, 637)$, $(28, 182)$, $(672, 273)$, $(168, 91)$, $(448, 546)$, $(560, 455)$, $(476, 546)$, $(28, 273)$, $(588, 455)$, $(448, 637)$ & 14 \\ \hline 
$(0, 0)$, $(0, 91)$, $(0, 182)$, $(0, 273)$, $(28, 0)$, $(56, 0)$, $(448, 364)$, $(476, 364)$, $(168, 0)$, $(560, 364)$, $(588, 364)$, $(672, 364)$, $(672, 91)$, $(588, 182)$, $(168, 637)$, $(28, 91)$, $(588, 273)$, $(56, 91)$, $(448, 455)$, $(476, 455)$, $(168, 91)$, $(560, 455)$, $(588, 455)$, $(168, 182)$ & 42 \\ \hline 
$(0, 0)$, $(0, 91)$, $(0, 182)$, $(0, 273)$, $(28, 0)$, $(448, 364)$, $(168, 0)$, $(560, 364)$, $(252, 0)$, $(672, 364)$, $(700, 364)$, $(560, 182)$, $(504, 273)$, $(252, 546)$, $(672, 182)$, $(588, 273)$, $(700, 182)$, $(252, 637)$, $(476, 455)$, $(28, 182)$, $(448, 546)$, $(56, 273)$, $(252, 91)$, $(168, 182)$ & 14 \\ \hline 
$(0, 0)$, $(0, 91)$, $(0, 182)$, $(0, 273)$, $(28, 0)$, $(476, 364)$, $(168, 0)$, $(560, 364)$, $(588, 364)$, $(280, 0)$, $(700, 364)$, $(560, 182)$, $(476, 273)$, $(672, 91)$, $(168, 637)$, $(560, 273)$, $(672, 182)$, $(588, 273)$, $(280, 637)$, $(700, 273)$, $(168, 91)$, $(560, 455)$, $(28, 273)$, $(168, 182)$ & 14 \\ \hline 
$(0, 0)$, $(0, 91)$, $(0, 182)$, $(0, 273)$, $(28, 0)$, $(56, 0)$, $(476, 364)$, $(168, 0)$, $(560, 364)$, $(280, 0)$, $(280, 455)$, $(476, 273)$, $(700, 91)$, $(252, 546)$, $(168, 637)$, $(280, 546)$, $(560, 273)$, $(700, 182)$, $(280, 637)$, $(168, 91)$, $(28, 273)$, $(56, 273)$, $(252, 91)$, $(168, 182)$ & 14 \\ \hline 
$(0, 0)$, $(0, 91)$, $(0, 182)$, $(0, 273)$, $(28, 0)$, $(56, 0)$, $(504, 364)$, $(168, 0)$, $(560, 364)$, $(588, 364)$, $(252, 0)$, $(672, 364)$, $(560, 182)$, $(476, 273)$, $(588, 182)$, $(700, 91)$, $(252, 546)$, $(56, 91)$, $(700, 182)$, $(476, 455)$, $(28, 182)$, $(700, 273)$, $(504, 546)$, $(56, 273)$ & 14 \\
\hline  
\end{longtable}}
The above two-intersection sets give rise to strongly regular graphs with parameters $(729, 336, 153, 156)$.

{\footnotesize
\begin{longtable}{|p{0.85\textwidth}|p{0.1\textwidth}|}
\caption{Two-intersection sets of type $(182, 6, 56, 65)$} \\ 
\hline
\textbf{Segre variety representation} & \textbf{Stabiliser} \\
& \textbf{size}  \\\hline
\endhead
$(0, 0)$, $(0, 91)$, $(0, 182)$, $(0, 273)$, $(28, 0)$, $(448, 364)$, $(168, 0)$, $(560, 364)$, $(476, 273)$, $(252, 546)$, $(168, 637)$, $(28, 91)$, $(672, 182)$, $(700, 182)$, $(252, 637)$, $(448, 455)$, $(28, 182)$, $(672, 273)$, $(56, 182)$, $(700, 273)$, $(168, 91)$, $(560, 455)$, $(476, 546)$, $(28, 273)$, $(56, 273)$, $(168, 182)$ & 14 \\ \hline 
$(0, 0)$, $(0, 91)$, $(0, 182)$, $(0, 273)$, $(28, 0)$, $(56, 0)$, $(448, 364)$, $(476, 364)$, $(504, 364)$, $(168, 0)$, $(560, 364)$, $(588, 364)$, $(252, 0)$, $(280, 0)$, $(672, 364)$, $(700, 364)$, $(280, 455)$, $(476, 273)$, $(672, 91)$, $(504, 273)$, $(168, 637)$, $(672, 182)$, $(588, 273)$, $(672, 273)$, $(168, 91)$, $(168, 182)$ & 21 \\ \hline 
$(0, 0)$, $(0, 91)$, $(0, 182)$, $(0, 273)$, $(28, 0)$, $(56, 0)$, $(476, 364)$, $(504, 364)$, $(168, 0)$, $(560, 364)$, $(588, 364)$, $(280, 0)$, $(280, 455)$, $(588, 182)$, $(252, 546)$, $(672, 182)$, $(588, 273)$, $(56, 91)$, $(700, 182)$, $(56, 182)$, $(700, 273)$, $(168, 91)$, $(448, 546)$, $(560, 455)$, $(588, 455)$, $(56, 273)$ & 21 \\ \hline 
$(0, 0)$, $(0, 91)$, $(0, 182)$, $(0, 273)$, $(28, 0)$, $(56, 0)$, $(168, 0)$, $(560, 364)$, $(588, 364)$, $(280, 455)$, $(560, 182)$, $(588, 182)$, $(504, 273)$, $(28, 91)$, $(588, 273)$, $(56, 91)$, $(280, 637)$, $(476, 455)$, $(28, 182)$, $(56, 182)$, $(168, 91)$, $(560, 455)$, $(588, 455)$, $(56, 273)$, $(252, 91)$, $(168, 182)$ & 14 \\ \hline 
$(0, 0)$, $(0, 91)$, $(0, 182)$, $(0, 273)$, $(28, 0)$, $(56, 0)$, $(476, 364)$, $(168, 0)$, $(560, 364)$, $(672, 364)$, $(476, 273)$, $(672, 91)$, $(504, 273)$, $(168, 637)$, $(280, 546)$, $(28, 91)$, $(56, 91)$, $(280, 637)$, $(476, 455)$, $(168, 91)$, $(448, 546)$, $(560, 455)$, $(476, 546)$, $(504, 546)$, $(168, 182)$, $(448, 637)$ & 14 \\
\hline  
\end{longtable}}
The above examples give rise to strongly regular graphs with parameters $(729, 364, 181, 182)$.
The example with stabiliser of order $20160$ arises by taking the derived subgroup of the stabiliser of the Hill-cap. This group has orbit lengths $56$, $56$, $126$, $126$ on the points of $\PG(5,3)$.
Taking a union of an orbit of size $56$ with one of size $126$ yields a two-intersection set with parameters $(182, 6, 56, 65)$.
The five examples in the table above that have an asterisk in the last column arise from partial spreads of $\PG(5,3)$; the `SU2' construction in \cite[\S7]{CalderbankSurvey}. Two of these examples arise from the `CY2' construction in \cite[\S9]{CalderbankSurvey}, denoted by two asterisks.

\section*{Acknowledgements}

First the second author would like to acknowledge Pablo Spiga for invaluable discussions on
research related to this work. In particular, there are some amorphic $3$-class association schemes
related to some of the codes found in our search.
We would like to thank Gordon Royle and Michael Giudici for their invaluable time and for the use of their computational facilities. 
Finally, we also acknowledge Andrew Bassom for supporting the first author in a Summer Vacation Scholarship.


\begin{thebibliography}{10}

\bibitem{Bouyukliev}
Iliya Bouyukliev.
\newblock On the binary projective codes with dimension 6.
\newblock {\em Discrete Appl. Math.}, 154(12):1693--1708, 2006.

\bibitem{CalderbankSurvey}
R.~Calderbank and W.~M. Kantor.
\newblock The geometry of two-weight codes.
\newblock {\em Bull. London Math. Soc.}, 18(2):97--122, 1986.

\bibitem{DelsarteSRG}
Ph. Delsarte.
\newblock Weights of linear codes and strongly regular normed spaces.
\newblock {\em Discrete Math.}, 3:47--64, 1972.

\bibitem{GAP4}
The GAP~Group.
\newblock {\em {GAP -- Groups, Algorithms, and Programming, Version 4.6.2}},
  2013.

\bibitem{minion}
I.~P. Gent, C.~Jefferson, and I.~Miguel.
\newblock {MINION:} a fast, scalable, constraint solver.
\newblock {\em The European Conference on Artificial Intelligence}, ECAI 06,
  2006.

\bibitem{Gulliver2}
T.~A. Gulliver.
\newblock A new two-weight code and strongly regular graph.
\newblock {\em Appl. Math. Lett.}, 9(2):17--20, 1996.

\bibitem{Gulliver1}
T.~Aaron Gulliver.
\newblock Two new optimal ternary two-weight codes and strongly regular graphs.
\newblock {\em Discrete Math.}, 149(1-3):83--92, 1996.

\bibitem{Hill73}
Raymond Hill.
\newblock On the largest size of cap in {$S_{5,\,3}$}.
\newblock {\em Atti Accad. Naz. Lincei Rend. Cl. Sci. Fis. Mat. Natur. (8)},
  54:378--384 (1974), 1973.

\bibitem{HirschfeldThas}
J.~W.~P. Hirschfeld and J.~A. Thas.
\newblock {\em General {G}alois geometries}.
\newblock Oxford Mathematical Monographs. The Clarendon Press, Oxford
  University Press, New York, 1991.
\newblock Oxford Science Publications.

\bibitem{gurobi}
Gurobi~Optimization Inc.
\newblock Gurobi optimizer version 5.0.
\newblock \url{http://www.gurobi.com/}.

\bibitem{KohnertCodes}
Axel Kohnert.
\newblock Constructing two-weight codes with prescribed groups of
  automorphisms.
\newblock {\em Discrete Appl. Math.}, 155(11):1451--1457, 2007.

\bibitem{nauty}
B.~D. McKay and A.~Piperno.
\newblock Practical graph isomorphism, {II}.
\newblock {\em J. Symbolic Computation}, 60:94--112, 2013.

\bibitem{RoyleParameters}
Tim Penttila and Gordon~F. Royle.
\newblock Sets of type {$(m,n)$} in the affine and projective planes of order
  nine.
\newblock {\em Des. Codes Cryptogr.}, 6(3):229--245, 1995.

\bibitem{VanDamEtAl}
Edwin~R. van Dam, William~J. Martin, and Mikhail Muzychuk.
\newblock Uniformity in association schemes and coherent configurations:
  cometric {Q}-antipodal schemes and linked systems.
\newblock {\em J. Combin. Theory Ser. A}, 120(7):1401--1439, 2013.

\end{thebibliography}

\end{document}